\documentclass[a4paper, 11pt]{amsart} 

\usepackage[letterpaper,margin=1in]{geometry}
\usepackage{etex}
\usepackage{amsmath,amssymb,amsthm,amsfonts,mathrsfs, mathtools}
\usepackage[frame,cmtip,arrow,matrix,line,graph,curve]{xy}
\usepackage{graphpap,color,paralist,pstricks}
\usepackage[mathscr]{eucal}
\usepackage{mathabx}
\usepackage[pdftex,colorlinks,backref=page,citecolor=blue]{hyperref}
\usepackage{tikz}
\usetikzlibrary{calc,decorations.markings}
\usepackage{epic,eepic}
\usepackage{yfonts}
\usepackage{enumitem} 
\usepackage{bbm}
\usepackage{tikz-cd}
\usepackage{aliascnt}
\usepackage{mathbbol}

\allowdisplaybreaks

\usepackage{xcolor, color, soul}

\usepackage{color}

\renewcommand{\leq}{\leqslant}
\renewcommand{\geq}{\geqslant}

\newtheorem{theorem}{Theorem}[section]

\newaliascnt{headcor}{headthm}

\aliascntresetthe{headcor}

\newaliascnt{headconj}{headthm}

\aliascntresetthe{headconj}

\newaliascnt{corollary}{theorem}
\newtheorem{corollary}[corollary]{Corollary}
\aliascntresetthe{corollary}

\newaliascnt{claim}{theorem}

\aliascntresetthe{claim}

\newaliascnt{lemma}{theorem}
\newtheorem{lemma}[lemma]{Lemma}
\aliascntresetthe{lemma}

\newaliascnt{thmdfn}{theorem}
\newtheorem{thmdfn}[thmdfn]{Theorem/Definition}
\aliascntresetthe{thmdfn}

\newaliascnt{conjecture}{theorem}

\aliascntresetthe{conjecture}

\newaliascnt{proposition}{theorem}
\newtheorem{proposition}[proposition]{Proposition}
\aliascntresetthe{proposition}

\theoremstyle{definition}
\newaliascnt{definition}{theorem}
\newtheorem{definition}[definition]{Definition}
\aliascntresetthe{definition}

\newaliascnt{notation}{theorem}

\aliascntresetthe{notation}

\newaliascnt{example}{theorem}
\newtheorem{example}[example]{Example}
\aliascntresetthe{example}

\newaliascnt{examples}{theorem}

\aliascntresetthe{examples}

\newaliascnt{remark}{theorem}
\newtheorem{remark}[remark]{Remark}
\aliascntresetthe{remark}

\newaliascnt{fact}{theorem}

\aliascntresetthe{fact}

\newaliascnt{question}{theorem}

\aliascntresetthe{question}

\newaliascnt{questions}{theorem}

\aliascntresetthe{questions}

\newaliascnt{problem}{theorem}

\aliascntresetthe{problem}

\newaliascnt{construction}{theorem}

\aliascntresetthe{construction}

\newaliascnt{setup}{theorem}

\aliascntresetthe{setup}

\newaliascnt{algorithm}{theorem}

\aliascntresetthe{algorithm}

\newaliascnt{observation}{theorem}

\aliascntresetthe{observation}

\newaliascnt{discussion}{theorem}

\aliascntresetthe{discussion}

\newaliascnt{defprop}{theorem}

\aliascntresetthe{defprop}

\def\sectionautorefname~#1\null{Section #1\null}
\def\subsectionautorefname~#1\null{\S #1\null}

\newcommand{\widetildespan}[1]{%
	\tikz[baseline=(T.base)]{
		\node[inner sep=0pt](T){$\displaystyle#1$};
		\draw[line width=0.4pt]
		([yshift=2pt]T.north west)
		.. controls ([yshift=7pt,  xshift= 5pt]T.north west)
		and     ([yshift=2pt,  xshift=-5pt]T.north)
		.. ([yshift=5pt]T.north)
		.. controls ([yshift=8pt,  xshift= 5pt]T.north)
		and     ([yshift=3pt,  xshift=-5pt]T.north east)
		.. ([yshift=6pt]T.north east);
	}%
}

\newcommand{\Mat}{\mathrm{Mat}}

\DeclareMathOperator{\rk}{rk} %
\DeclareMathOperator{\sym}{sym} %
\DeclareMathOperator{\Supp}{Supp} %

\newcommand{\R}{\mathbb{R}} %
\newcommand{\Z}{{\mathbb{Z}}} %
\newcommand{\N}{\mathbb{Z}_{\geq 0}} %

\title[Matroid correspondence]{Matroid correspondence}
\author[Crowley, Ding, Hu, Kumar, Montaño,  Nguy$\tilde{\text{\^E}}$n, Shapiro, Solomon, Wang, Whidden]{%
	Colin Crowley, Changxin Ding, Haoxi Hu, Arvind Kumar, Jonathan Montaño,
	Th\'ai Th\`anh Nguy$\tilde{\text{\^E}}$n, Anna Shapiro, Noah Solomon, Botong Wang, and Juliet Whidden}

\address{Colin Crowley -- Department of Mathematics, University of Oregon}
\email{crowley@uoregon.edu}

\address{Changxin Ding -- School of Mathematics, Georgia Institute of Technology}
\email{cding66@gatech.edu}

\address{Haoxi Hu -- Department of Mathematics, Tulane University}
\email{hhu5@tulane.edu}

\address{Arvind Kumar -- Department of Mathematical Sciences, New Mexico State University}
\email{arvkumar@nmsu.edu}

\address{Jonathan Montaño -- School of Mathematical and Statistical Sciences, Arizona
	State University}
\email{montano@asu.edu}

\address{Th\'ai Th\`anh Nguy$\tilde{\text{\^E}}$n -- Department of Mathematics, University of Dayton}
\email{tnguyen5@udayton.edu}

\address{Anna Shapiro -- Department of Mathematics, North Carolina State University}
\email{arshapi4@ncsu.edu}

\address{Noah Solomon -- School of Mathematics, Georgia Institute of Technology}
\email{noah.solomon@math.gatech.edu}

\address{Botong Wang -- Department of Mathematics, University of Wisconsin--Madison}
\email{wang@math.wisc.edu}

\address{Juliet Whidden -- School of Mathematics, Georgia Institute of Technology}
\email{jwhidden6@gatech.edu}

\begin{document}

	\begin{abstract}
		Motivated by algebraic correspondences and linear operators associated with volume and Lorentzian polynomials, we introduce matroid correspondences and their polymatroid analogues. A matroid correspondence defines a functor between poset categories of matroids whose morphisms are matroid quotients, and various standard functors, including deletion, contraction, free extension, truncation, intersection,  union, and pullback, arise in this way. We show that these correspondences preserve representability and algebraicity under natural hypotheses. In the polymatroid setting, we establish compatibility with multisymmetric lifts. Finally, we relate this construction to the supports of linear operators with Lorentzian symbols.
	\end{abstract}

	%\keywords{}
	%\subjclass[2020]{Primary:  Secondary: }
	
	\maketitle
	
	\section{Introduction}
	
	Motivated by the recent work on realization problems and Lorentzian/volume polynomials, see e.g. \cite{BH20, multidegrees, realization, GHM+25}, we observe that an irreducible subvariety of a product of projective spaces determines a realizable volume polynomial, and taking support gives rise to an algebraic polymatroid. Thus, we have a diagram
	\[
	\{\text{subvarieties of } \prod_i \mathbb{P}^{n_i}\}
	\longrightarrow
	\{\text{volume polynomials}\}
	\longrightarrow
	\{\text{polymatroids}\},
	\]
	where in the first arrow we forget the geometry of the variety and only remember its homology class, and in the second arrow we forget the coefficients of the polynomial and only remember its support. 
	
	In \cite{GHM+25}, linear operators of volume polynomials were studied extensively. Geometrically, linear operators on volume polynomials are the analogues of an algebraic correspondence; see \cite[Chapter 16]{Fulton} for an introduction. In this note, we introduce the corresponding combinatorial construction on the  matroid/polymatroid side, which we call \emph{matroid/polymatroid correspondence}. By definition, these correspondences are functors between suitable poset categories.

	Let $E$ be a finite set, and let $M$ and $N$ be two matroids with ground set $E$. We say that $N$ is a \emph{quotient} of $M$, denoted by $M\twoheadrightarrow N$, if any flat of $N$ is also a flat of $M$.  %Equivalently, $N$ is a quotient of $M$ if the identity map on $E$ is a {\it strong map} from $M$ to $N$. 
	We denote by $\Mat(E)$ the \textit{poset category of matroids} on $E$, whose objects are the set of matroids with ground set $E$, and whose morphisms are matroid quotients. Note that there is at most one morphism between any two matroids in the category $\Mat(E)$. Hence,  $\Mat(E)$ can  be viewed as a poset of matroids on $E$.

	\begin{definition}\label{def:mat_corr}
		Let  $E_1,E_2$ be finite sets and $\mathscr{C}\in \Mat(E_1 \sqcup E_2)$. We define the \emph{matroid correspondence} ${\mathscr{C}}_*$ to be the functor
		\begin{align*}
			{\mathscr{C}}_*: \Mat(E_1) &\to\Mat(E_2)\\
			M &\mapsto ((M\oplus \mathbf{B}_{E_2}) \wedge {\mathscr{C}})\setminus E_1,
		\end{align*}
		where $\mathbf{B}_{E_2}$ denotes the boolean matroid on $E_2$, and $\wedge$ denotes matroid intersection. In \autoref{prop:functor in intro} we show that $\mathscr{C}_*$ is indeed a functor.  
	\end{definition}

	We show that various standard functors on matroids, including deletion, contraction, free extension, truncation, intersection,  union, and pullback, arise as correspondences, see \autoref{thm:matroid example}, \autoref{cor:trunctation}, and \autoref{thm:pullback}. In \autoref{prop:algebraic and representable} 
	we also show that if a matroid $M$ is representable or algebraic, then so is $\mathscr{C}_*(M)$. Moreover, composition of matroid correspondences is again a matroid correspondence, as shown in  \autoref{prop:composition}.

	%Our definition of matroid correspondences naturally generalizes to polymatroids. 
	In the last part of the paper, we extend our theory to polymatroids. In particular, we prove the following statements.
	\begin{enumerate}[leftmargin=24pt]
		\item Polymatroid correspondence commutes with multisymmetric lift (\autoref{prop:commute with lift}).
		\item The polymatroid correspondence defined by an algebraic polymatroid preserves algebraic polymatroids (\autoref{prop:polymatroid intersection preserves algebraic}). 
		\item Upon replacing each polynomial with its support, a linear operator on polynomials (with Lorentzian symbol) gives rise to polymatroid correspondence (\autoref{prop:polymatroid correspondence as support}).
	\end{enumerate}
	
	Although matroid/polymatroid correspondences serve as combinatorial analogues of algebraic correspondences and linear operators on Lorentzian or volume polynomials, there is a key difference. The image of an algebraic correspondence or a linear operator may be zero. Combinatorially, this would correspond to a matroid or polymatroid with an empty set of bases, which is impossible by definition.
	
	In fact, when the algebraic correspondence or linear operator has zero image, the associated polymatroid correspondence instead produces a matroid or polymatroid of strictly larger rank than expected. One way to observe this phenomenon is through matroid intersection: one might expect that the corank of the intersection equals the sum of the coranks of the two matroids, but this is not always the case; see \cite[Proof of Theorem 5.11]{GHM+25} for a remedy of the rank discrepancy using Higgs lifts. %\footnote{See \cite[Proof of Theorem 5.11]{GHM+25} for a remedy of the rank discrepancy using Higgs lifts.} 
	This also explains why better functorial properties arise when the rank behaves well; see \autoref{prop:uniform_preserving_corrs}. 
	
	Thus, matroid and polymatroid correspondences are, in a sense, more forgiving: they continue to yield meaningful combinatorial information even when the corresponding algebraic correspondence or linear operator has zero image.
	
	\subsection*{Notation and organization of the paper} 
	All matroids considered in this paper are over finite sets. For a matroid $M$ on $E$ we denote its dual matroid by $M^\vee$ and its set of  bases by $\mathcal{B}_M$. If $N$ is another matroid on $E$, we write $M\wedge N$ and $M\vee N$ for their intersection and union, respectively, see \cite[Definition 5.10]{GHM+25}. We write $M \twoheadrightarrow N$ if $N $ is a quotient of $M$. We denote by $\Mat(E)$ the category of matroids on $E$ with matroid quotients as morphisms. 	
	
	For any set $E$ we write $E^*$ for a set disjoint from $E$ and equipped with a natural bijection  $e\mapsto e^*$, $e\in E$. If $E$ has $n$ elements, we denote by $U_{r,n}$ or $U_{r,E}$ the uniform matroid on $E$ of rank $r$.
	
	We refer the reader to standard references on matroid theory for the remaining unexplained notation and terminology \cite{Whi86,Oxley}.

	In \autoref{sec:functor}, we show that matroid correspondences are functors and some other properties. In \autoref{sec:example and properties}, we give several examples of functors that can be realized as matroid correspondences. We also investigate matroid correspondences where $E_1$ and $E_2$ have the same cardinality. % and pose several questions.
	
	The definition, standard constructions, and basic properties of polymatroids are reviewed in \autoref{sec:background on polymatroids}. In \autoref{sec:Lorentzian and volume}, we discuss the relationship between polymatroids and Lorentzian/volume polynomials. The definition and various properties of polymatroid correspondences are introduced in \autoref{sec:polynmatroid correspondences}.
	
	Finally, in \autoref{sec:relation with symbols}, we show that polymatroid correspondences are the combinatorial analogues of linear operators on Lorentzian polynomials.
	
	\subsection*{Acknowledgments}
	This paper grew out of discussions at the Arizona--New Mexico Symposium on Commutative Algebra and its Interactions: Geometric Combinatorics, held at New Mexico State University on December 8--12, 2025. The authors thank the organizers and participants of the symposium for providing a stimulating environment for collaboration, and New Mexico State University for hosting the event. The symposium was supported by National Science Foundation grants DMS--2401522 and  DMS--2344588, the Department of Mathematics at New Mexico State University, and the School of Mathematical and Statistical Sciences at Arizona State University. The second, fourth, and sixth authors were partially funded by the AMS--Simons Travel Grant. The fifth author was partially funded by NSF Grant DMS--2401522. The eighth and tenth authors were partially supported by Department of Education Graduate Assistance in Areas of National Need, Award P200A240169.
	
	\section{Properties of matroid correspondences}\label{sec:functor}
	In this section, we prove that matroid correspondences are functors, see \autoref{prop:functor in intro}, and we provide some other basic properties.
	
	First, we show that direct sum and deletion both preserve matroid quotients. 
	
	\begin{lemma}\label{lem:sum and deletion}
		Suppose $M,N\in\Mat(E_1)$, $P\in\Mat(E_2)$, and $E_3\subseteq E_1$. 
		\begin{enumerate}
			\item If $M\twoheadrightarrow N$, then $M\oplus P \twoheadrightarrow N\oplus P$.
			\item If $M\twoheadrightarrow N$, then $M\setminus E_3 \twoheadrightarrow N\setminus E_3$.
		\end{enumerate}   
	\end{lemma}
	\begin{proof}
		Recall that the identity on $E$ induces a matroid quotient $M\twoheadrightarrow N$ if and only if every flat in $N$ is a flat in $M$. This property is readily shown to be preserved when taking direct sums and  deletions. 
	\end{proof}

	%\begin{definition}\label{def:mat_intersection}
	%For $M, N\in \Mat(E)$, the \emph{matroid intersection} $M\vee N$ is the matroid on $E$ whose spanning sets are the intersections of spanning sets of $M$ and $N$.
	%\end{definition}
	
	We now show that matroid intersection preserves matroid quotients. %We prove this using matroid duality.
	%xFor a matroid $M$, we denote its dual matroid by $M^\vee$. 
	%\begin{lemma}
	%\label{lem: dual}
	%Let $M,N\in\Mat(E)$. Then $M\twoheadrightarrow N$ if and only if $N^\vee\twoheadrightarrow M^\vee$. 
	%\end{lemma}

	\begin{lemma}\label{lem:intersection}
		Suppose $M_1, M_2, N_1, N_2\in\Mat(E)$. 
		\begin{enumerate}
			\item If $M_1\twoheadrightarrow N_1$ and $M_2\twoheadrightarrow N_2$, then $M_1\vee M_2\twoheadrightarrow N_1\vee N_2$.
			\item If $M_1\twoheadrightarrow N_1$ and $M_2\twoheadrightarrow N_2$, then $M_1\wedge M_2\twoheadrightarrow N_1\wedge N_2$. 
		\end{enumerate}
	\end{lemma}
	
	\begin{proof}  
		The first part is \cite[Exercise 8.7]{Whi86}. For the second part, we first get $N_1^\vee\twoheadrightarrow M_1^\vee$ and $N_2^\vee\twoheadrightarrow M_2^\vee$ by \cite[Proposition 7.3.1]{Oxley}. Then by the first part, we have $N_1^\vee\vee N_2^\vee\twoheadrightarrow M_1^\vee\vee M_2^\vee$. Then taking the dual again, we get $M_1\wedge M_2\twoheadrightarrow N_1\wedge N_2$. 
	\end{proof}
	
	With the previous two lemmas in hand, we are ready to show that matroid correspondences define functors, see \autoref{def:mat_corr}.
	
	\begin{proposition}\label{prop:functor in intro}
		Let $M,N\in \Mat(E_1)$  with $M\twoheadrightarrow N$. For every  $\mathscr{C}\in \Mat(E_1 \sqcup E_2)$ we have a matroid quotient 
		$
		\mathscr{C}_*(M)
		\to
		\mathscr{C}_*(N)
		$ 
		between matroids on $E_2$.
	\end{proposition}
	\begin{proof}%[Proof of Proposition \ref{prop:functor in intro}]
		%Let $\mathscr{C}$ be a matroid on $E_1 \sqcup E_2$ and $M,N\in\Mat(E_1)$. 
		%We need to prove that $M\twoheadrightarrow N$ implies $\mathscr{C}_*(M)\twoheadrightarrow \mathscr{C}_*(N)$. 
		Combining \autoref{lem:sum and deletion} (1) and \autoref{lem:intersection} (2) we obtain \[(M\oplus \mathbf{B}_{E_2}) \wedge \mathscr{C}\twoheadrightarrow (N\oplus \mathbf{B}_{E_2}) \wedge \mathscr{C}.\]
		The conclusion now follows from \autoref{lem:sum and deletion} (2).
	\end{proof}
	
	Although matroid correspondences are functorial for matroid quotients, they need not be so for other natural maps on matroids, such as weak maps,   as the following example shows. Recall that if $M, N\in \Mat(E)$, we say $M\to N$ is a \emph{weak map} if the bases of $N$ are independent in $M$.

	\begin{example}\label{ex:weak maps not preserved}
		We now give an example of a correspondence that does not preserve weak maps.
		
		Consider the matroid $\mathscr{C}$ on $E \sqcup E^* = \{1,2\}\sqcup \{1^*,2^*\}$ with bases given by
		\begin{align*}
			\mathcal{B}_{\mathscr{C}}\coloneqq \{\{1,1^*\},\{1^*,2^*\}, \{1, 2^*\}\}.
		\end{align*}
		%In other words, $\mathscr{C}$ is the direct sum of the rank zero matroid on $\{2\}$ and the rank 2 uniform matroid on $\{1, 1^*, 2^*\}$. 
		Let $M$ and $N$ be matroids on $\{1,2\}$ whose bases are given by
		\[
		\mathcal{B}_M\coloneqq \{\{1\}, \{2\}\}\quad \text{and}\quad \mathcal{B}_N\coloneqq \{\{1\}\}.
		\]
		%Let $M$ denote $U_{1,2}$ on $\{1,2\}$ and $N$ denote $U_{1,1}\oplus U_{0,1}$ with single basis $\{1\}$. 
		Then one computes the following:
		\begin{align*}
			\mathcal{B}_{\mathscr{C}_*(M)} = \{\{1^*\}, \{2^*\}\}\quad \text{and}\quad
			\mathcal{B}_{\mathscr{C}_*(N)} = \{\{1^*,2^*\}\}.
		\end{align*}
		Therefore, $M\to N$ is a weak map, yet $\mathscr{C}_*(M)\to  \mathscr{C}_*(N)$ is not. %Therefore, correspondences do not preserve weak maps in general.
	\end{example}
	
	In the following proposition we  show that if a matroid $M$ is representable or algebraic, then so is $\mathscr{C}_*(M)$.
	
	\begin{proposition}\label{prop:algebraic and representable}\hfill
		\begin{enumerate}
			\item Given any field $\mathbb{K}$, if ${\mathscr{C}}$ is an algebraic matroid over $\mathbb{K}$ on $E_1\sqcup E_2$, then ${\mathscr{C}}_*$ preserves the algebraic matroids over $\mathbb{K}$. 
			\item Given any infinite field $\mathbb{K}$, if ${\mathscr{C}}$ is a representable matroid over $\mathbb{K}$ on $E_1\sqcup E_2$, then ${\mathscr{C}}_*$ preserves the representable matroids over $\mathbb{K}$. 
		\end{enumerate}
	\end{proposition}
	\begin{proof}
		For the first statement, notice that taking direct sum and deletion both preserve algebraic matroids over $\mathbb{K}$. Moreover, by \cite[Theorem 5.11]{GHM+25}, the intersection of two algebraic matroids over $\mathbb{K}$ is also an algebraic matroid over $\mathbb{K}$. Therefore, if $\mathscr{C}$ is algebraic over $\mathbb{K}$, then %each operation in the definition of $\mathscr{C}_*$ preserves algebraic matroids over $\mathbb{K}$. Thus, 
		$\mathscr{C}_*$ preserves algebraic matroids over $\mathbb{K}$. 
		
		For the second statement, %we use the same argument. Taking 
		similarly we observe that direct sum and deletion both preserve representable matroids over $\mathbb{K}$. By a theorem of Piff and Welsh \cite[Page 435, Exercise 9]{Oxley} together with the preservation of representability under duality \cite[Corollary 2.2.9]{Oxley}, when $\mathbb{K}$ is an infinite field, taking intersection preserves representability over $\mathbb{K}$. Thus, if $\mathscr{C}$ is representable over $\mathbb{K}$, then $\mathscr{C}_*$ preserves representable matroids over $\mathbb{K}$. 
	\end{proof}
	
	The following proposition is a combinatorial analogue of the product correspondence formula; see \cite[Definition 16.1.1]{Fulton}.
	
	\begin{proposition}\label{prop:composition}
		Composition of matroid correspondences is also a matroid correspondence. More precisely,  	given disjoint finite sets $E_1, E_2, E_3$, and matroids $\mathscr{C}_1 \in \Mat(E_1\sqcup E_2)$, 
		$\mathscr{C}_2 \in \Mat(E_2\sqcup E_3)$, the  composition
		\[
		\mathscr{C}_{2*}\circ \mathscr{C}_{1*}: \Mat(E_1)\to \Mat(E_3)
		\]
		is equal to $\mathscr{C}_{*}$, where
		\[
		\mathscr{C}=\big((\mathscr{C}_{1}\oplus \mathbf{B}_{E_3})\wedge (\mathbf{B}_{E_1}\oplus \mathscr{C}_{2})\big)\setminus E_2.
		\]
	\end{proposition}

	\begin{proof}
		The proof follows by direct computation. For every  $M\in \Mat(E_1)$ we have 
		\begin{align*}
			\Big(\Big(\big((M\oplus \mathbf{B}_{E_2})\wedge \mathscr{C}_1\big)\setminus E_1\Big)\oplus \mathbf{B}_{E_3}\Big)\wedge \mathscr{C}_2
			&=\Big(\Big(\big((M\oplus \mathbf{B}_{E_2})\wedge \mathscr{C}_1\big)\oplus \mathbf{B}_{E_3}\Big)\setminus E_1\Big)\wedge \mathscr{C}_2\\
			&=\Big(\Big(\big((M\oplus \mathbf{B}_{E_2})\wedge \mathscr{C}_1\big)\oplus \mathbf{B}_{E_3}\Big)\wedge (\mathbf{B}_{E_1}\oplus \mathscr{C}_2)\Big)\setminus E_1\\
			&=\Big(\Big((M\oplus \mathbf{B}_{E_2}\oplus \mathbf{B}_{E_3})\wedge \big(\mathscr{C}_1\oplus \mathbf{B}_{E_3}\big)\Big)\wedge (\mathbf{B}_{E_1}\oplus \mathscr{C}_2)\Big)\setminus E_1.
		\end{align*}
		Therefore, 
		\[
		\mathscr{C}_{2*}\circ \mathscr{C}_{1*}(M)=\Big((M\oplus \mathbf{B}_{E_2}\oplus \mathbf{B}_{E_3})\wedge \big(\mathscr{C}_1\oplus \mathbf{B}_{E_3}\big)\wedge (\mathbf{B}_{E_1}\oplus \mathscr{C}_2)\Big)\setminus (E_1\sqcup E_2).
		\]
		On the other hand, 
		\begin{align*}
			\mathscr{C}_{*}(M)&=\Big((M\oplus \mathbf{B}_{E_3})\wedge \Big(\big((\mathscr{C}_{1}\oplus \mathbf{B}_{E_3})\wedge (\mathbf{B}_{E_1}\oplus \mathscr{C}_{2})\big)\setminus E_2\Big)\Big)\setminus E_1\\
			&=\Big(\Big((M\oplus \mathbf{B}_{E_2}\oplus \mathbf{B}_{E_3})\wedge \big((\mathscr{C}_{1}\oplus \mathbf{B}_{E_3})\wedge (\mathbf{B}_{E_1}\oplus \mathscr{C}_{2})\big)\Big)\setminus E_2\Big)\setminus E_1\\
			&=\Big((M\oplus \mathbf{B}_{E_2}\oplus \mathbf{B}_{E_3})\wedge \big(\mathscr{C}_1\oplus \mathbf{B}_{E_3}\big)\wedge (\mathbf{B}_{E_1}\oplus \mathscr{C}_2)\Big)\setminus (E_1\sqcup E_2).
		\end{align*}
		Thus, 
		$
		\mathscr{C}_{2*}\circ \mathscr{C}_{1*}=\mathscr{C}_{*}
		$. 
	\end{proof}
	
	\section{Examples and further properties of matroid correspondences}\label{sec:example and properties}
	%identical but distinguished copy of the original set.
	In this section we provide some important examples of matroid correspondences. We begin with the following theorem.
	
	\begin{theorem}\label{thm:matroid example}
		Let $E$ be a finite set. The following functors on matroids  can be realized as correspondences via the given matroid. 
		\begin{enumerate}%[leftmargin=12pt]
			\item \textbf{Identity:} define the matroid $\mathscr{C}_{\mathrm{id}}$ on $E \sqcup E^*$ with bases 
			\begin{align*}
				\{S\sqcup (E\setminus S)^*\mid S\subseteq E\}.
			\end{align*}
			Then $\mathscr{C}_{\mathrm{id}*}(M)=M^*$ for every $M\in \Mat(E)$.\smallskip
			
			\item \textbf{Permutation of the ground set:} for any permutation $\sigma$ of $E$ define the matroid $\mathscr{C}_{\mathrm{\sigma}}$ on $E \sqcup E^*$ with bases 
			\begin{align*}
				\{S\sqcup (E\setminus \sigma(S))^*\mid S\subseteq E\}.
			\end{align*}
			Then $\mathscr{C}_{\mathrm{\sigma}*}(M)=\mathrm{\sigma}(M)^*$ for every $M\in \Mat(E)$, where $\mathrm{\sigma}(M)$ is the matroid on $E$ whose bases are $\{\mathrm{\sigma}(B)\mid B\in \mathcal{B}_M\}$.\smallskip
			
			\item \textbf{Deletion:} for any $e \in E$,  define the matroid $\mathscr{C}_{\setminus e}:= \mathscr{C}_{\mathrm{id}}\setminus e^*$ on $E \sqcup (E\setminus \{e\})^*$ with bases 
			\begin{align*}
				\{S\sqcup (E\setminus S)^*\mid e \in S\subseteq E\}.
			\end{align*}
			Then $\mathscr{C}_{\setminus e *}(M)=(M\setminus e)^*$ for every $M\in \Mat(E)$.\smallskip
			
			\item \textbf{Contraction:} for any $e \in E$, define the matroid $\mathscr{C}_{/e}:=\mathscr{C}_{\mathrm{id}}/ e^*$  on $E \sqcup (E\setminus \{e\})^*$ with bases 
			\begin{align*}
				\{S\sqcup (E\setminus (S\cup \{e\}))^*\mid S\subseteq E\setminus e\}.
			\end{align*}
			Then $\mathscr{C}_{/e *}(M)=(M/e)^*$ for every $M\in \Mat(E)$.\smallskip
			
			\item \textbf{Free extension:} for $t\not\in E$, define the matroid $\mathscr{C}_{+t}:=\mathscr{C}_{\mathrm{id}}+_{E\sqcup E^*}t$  on $E \sqcup E^*\sqcup \{t\} $ with bases 
			\begin{align*}
				\{S\sqcup (E\setminus S)^*\mid S\subseteq E\}\cup \{S\sqcup (E\setminus (S\cup\{e\}))^*\sqcup \{t\}\mid S\subseteq E, e\in E\setminus S\}.
			\end{align*}
			Then $\mathscr{C}_{+t *}(M)=(M+_E t)^*$ for every $M\in \Mat(E)$.\smallskip

			\item \textbf{Intersection with a fixed matroid:} for any matroid $M_0$ on $E$ define $\mathscr{C}_{\wedge M_0}$ on $E \sqcup E^*$ by 
			\[
			\mathscr{C}_{\wedge M_0}\coloneqq \mathscr{C}_{\mathrm{id}}\wedge (M_0 \oplus \mathbf{B}_{E^*}).
			\]
			Then $\mathscr{C}_{\wedge M_0*}(M)=(M\wedge M_0)^*$ for every $M\in \Mat(E)$.\smallskip
			
			\item \textbf{Union with a fixed matroid:} for any matroid $M_0$ on $E$ define $\mathscr{C}_{\vee M_0}$ on $E \sqcup E^*$ by 
			\[
			\mathscr{C}_{\vee M_0}\coloneqq \mathscr{C}_{\mathrm{id}}\vee (M_0 \oplus U_{0, E^*}).
			\]
			Then $\mathscr{C}_{\vee M_0*}(M)=(M\vee M_0)^*$ for every $M\in \Mat(E)$.\smallskip
			
			\item \textbf{Constant map:} let $E_1,E_2$ be finite sets. For any matroid $M_0$ on $E_2$  define 
			\[
			\mathscr{C}_{M_0}\coloneqq \mathbf{B}_{E_1} \oplus M_0.
			\]
			Then $\mathscr{C}_{M_0*}(M)=M_0$ for every $M\in \Mat(E_1)$.\smallskip
		\end{enumerate}
	\end{theorem}
	\begin{proof}
		The proof follows by direct computation of bases. We illustrate the strategy by providing the proofs of (5) and (7).
		
		For (5), we note that  the bases of $M\oplus \mathbf{B}_{E^*}$ are $\{B\sqcup E^*\mid B\in \mathcal{B}_M\}$. Therefore, the bases of 
		$(M\oplus \mathbf{B}_{E^*})\wedge \mathscr{C}_{+t}$ are the minimal sets with respect to inclusion in the collection 
		\begin{align*}
			\{(S\cap B)\sqcup (E\setminus S)^*\mid S\subseteq E, B\in \mathcal{B}_M\}\cup \{(S\cap B)\sqcup (E\setminus (S\cup\{e\}))^*\sqcup \{t\}\mid S\subsetneq E, e\in E\setminus S\}.
		\end{align*}
		In each of these sets, adding an element from $E\setminus B$ to $S$ only decreases the cardinality. Therefore, the minimal sets are  
		\begin{align*}
			\{S \sqcup (B\setminus S)^*\mid S\subseteq B, B\in \mathcal{B}_M\}
			\cup \{S \sqcup (B\setminus (S\cup\{e\}))^*\sqcup \{t\}\mid S\subsetneq B, e\in B\setminus S\}.
		\end{align*}
		It follows that the bases of  $\mathscr{C}_{+t *}(M)$ are the same as those of $(M+_E t)^*$.
		
		For (7), the set of bases of $\mathscr{C}_{\vee M_0}$ is
		$$
		\{
		(S\cup A)\sqcup (E\setminus S)^*\mid S\subseteq E, A\in\mathcal{B}_{M_0}
		\},
		$$
		that is, 
		$$
		\{
		(S\cup A)\sqcup (E\setminus S)^*\mid A\in\mathcal{B}_{M_0}, S\subseteq (E\setminus A)
		\}.
		$$
		Therefore, the bases of $(M\oplus \mathbf{B}_{E^*})\wedge \mathscr{C}_{\vee M_0}$ are the minimal sets of
		$$
		\{
		(B\cap (S\cup A))\sqcup (E\setminus S)^*\mid A\in\mathcal{B}_{M_0}, S\subseteq (E\setminus A),  B\in \mathcal{B}_M
		\}.
		$$
		A similar analysis as above shows that in each of those minimal sets one has $(E\setminus B)\cap (E\setminus A)\subset S$, that is, $(E\setminus (A\cup B))\subset S$. Thus, the minimal sets are 
		$$
		\{
		(B\cap ( (E\setminus S) \cup A))\sqcup S^*\mid B\in \mathcal{B}_M, A\in\mathcal{B}_{M_0}, A\subseteq S\subseteq (A\cup B)
		\}.
		$$
		It follows that the bases of  $\mathscr{C}_{\vee M_0*}(M)$ are the same as those of $(M\vee M_0)^*$.
	\end{proof}
	
	Let $M\in \Mat(E)$. The \emph{truncation} of $M$, denoted by $T(M)$, is the matroid on $E$ whose bases are the independent sets of $M$ of size $\rk(M)-1$. If $\rk(M)=0$ then $T(M)=M$. As a corollary we obtain that truncations can also be realized as correspondences.
	
	\begin{corollary}\label{cor:trunctation}
		Let $E$ be a finite set. Truncation on $\Mat(M)$ is a correspondence.	
	\end{corollary}
	\begin{proof}
		The result follows from \autoref{prop:composition} and \autoref{thm:matroid example} (4) and (5) since $T(M)=(M+_Et)/t$ for every $M$.
	\end{proof}
	
	For any function of finite sets $f:E_1\to E_2$ 	and  $N\in \Mat(E_2)$ one defines the \emph{pullback matroid}  
	$f^{-1}(N)\in \Mat(E_1)$ with rank function
	$\rk_{f^{-1}(N)}(S)=\rk_{N}(f(S))$ for  $S\subseteq E_1$;
	see \cite[\S2.2]{MatroidMorphism}. 
	
	In the following theorem we show that pullback matroids also arise as correspondences. 
	This  theorem is a combinatorial analogue of the algebraic correspondence induced by the graph of a morphism; see \cite[Proposition 16.1.2(c)]{Fulton}.
	
	\begin{theorem}\label{thm:pullback}
		Let $E_1$ and $E_2$ be finite sets and $f:E_1\to E_2$ a function. Define the matroid $\mathscr{C}_{f^{-1}}$ on $E_1\sqcup E_2$ with rank function $$\rk_{\mathscr{C}_{f^{}-1}}(S\sqcup T)=|f(S)\cup T|,\quad S\subseteq E_1, T\subseteq E_2.$$
		Then $\mathscr{C}_{f^{-1}*}(N)=f^{-1}(N)$ for every $N\in \Mat(E_2)$.   
	\end{theorem}
	\begin{proof}
		The bases of $\mathscr{C}_{f^{-1}}$ are given by the sets 
		$$
		\{
		S\sqcup (E_2\setminus f(S))
		\mid 
		S\subseteq E_1, f|_{S} \text{ is injective}
		\}.
		$$
		Therefore, the bases of $(\mathbf{B}_{E_1}\oplus N)\wedge \mathscr{C}_{f^{-1}}$ are the minimal sets with respect to inclusion in the collection 
		$$
		\{S\sqcup (B\cap (E_2\setminus f(S)))
		\mid 
		S\subseteq E_1, f|_{S} \text{ is injective}, B\in \mathcal{B}_N
		\}.
		$$
		In these sets, removing an element  $e$ from $S$ such that $f(e)\not\in B$ drops the cardinality of $S$ while keeping the one of  $B\cap (E_2\setminus f(S))$ the same. One concludes that the minimal sets are 
		$$
		\{S\sqcup (B\cap (E_2\setminus f(S)))
		\mid 
		S\subseteq E_1, f|_{S} \text{ is injective}, f(S)\subseteq B,  B\in \mathcal{B}_N
		\}.
		$$
		Thus, the independent sets of  $\mathscr{C}_{f^{-1}*}(N)$ are 
		$$
		\{
		S\subseteq E_1\mid f|_S\text{ is injective}, f(S)  \text{ is independent in }N
		\},
		$$
		which are the independent sets of $f^{-1}(N)$.
	\end{proof}
	
	In the next proposition we note that better functorial properties arise with better rank behavior under $\mathscr{C}_*$.

	\begin{proposition}\label{prop:uniform_preserving_corrs}
		Suppose $\mathscr{C}$ is a matroid on $E\sqcup E^*$ such that $\mathscr{C}_*$ fixes $U_{0,n}$ and $U_{n,n}$.
		\begin{enumerate}
			\item $\mathscr{C}_*$ is rank preserving;
			\item $\mathscr{C}_*$ preserves weak maps;
			\item $\mathscr{C}_*$ fixes all uniform matroids;
			\item The bases of $\mathscr{C}$ are determined by the image of $\mathscr{C}_*$ on matroids with exactly one basis.
		\end{enumerate}
	\end{proposition}
	\begin{proof}
		First we glean some information about the bases of $\mathscr{C}$ from which the rest  follows. Since $\mathscr{C}_*$ fixes $U_{0,n}$, ${\mathscr{C}}$ has a basis that is contained in $E$. Similarly, since $\mathscr{C}_*$ fixes $U_{n,n}$, ${\mathscr{C}}$ has a basis that contains $E^*$. Therefore, both $E$ and $E^*$ are bases of $\mathscr{C}$, and the rank of $\mathscr{C}$ is $n$. By the strong exchange property, we deduce that for any $S\subseteq E^*$ (resp. $S\subseteq E$), there exists a basis $B$ of $\mathscr{C}$ such that $B\cap E^*=S$ (resp. $B \cap E = S$).
		
		For (1), given any matroid $M$ with basis $S\subseteq E$, we can find a basis $B$ of $\mathscr{C}$ such that $B \cap E = E\setminus S$, so $B \cap E^*$ is be a basis of $\mathscr{C}_*(M)$ of equal cardinality to $S$.
		
		For $(2)$, we first show that $\mathscr{C}_*$ preserves rank-preserving weak maps. By our analysis in the first paragraph, we obtain:
		\begin{align}\label{eq:bases of nice correspondence}
			\mathcal{B}_{\mathscr{C}_*(M)} = \{T\subseteq E^*\mid \text{there exists}\; S \in \mathcal{B}_M \text{ such that } %(E\setminus S)\cap T=\emptyset, 
			(E\setminus S)\cup T \in \mathcal{B}_{\mathscr{C}}\}.
		\end{align}
		From this description, it follows that $\mathcal{B}_N\subseteq \mathcal{B}_M$ implies $\mathcal{B}_{\mathscr{C}_*(N)}\subseteq \mathcal{B}_{\mathscr{C}_*(M)}$. Thus, $\mathscr{C}_*$ preserves rank-preserving weak maps. By \cite[Proposition 4.8 (a)]{WeakMaps}, every weak map is the composition of rank-preserving weak maps and truncations. Since truncations are  preserved by $\mathscr{C}_*$, again by the description in \autoref{eq:bases of nice correspondence},  it follows that $\mathscr{C}_*$ preserves weak maps.
		
		For (3), given any set $S\subseteq E^*$ of size $r$, we may find a basis $B$ of $\mathscr{C}$ which agrees with $S$ on $E^*$. Then $B\cap E$ has size $n-r$, hence $E\setminus (B\cap E)$ is a basis of $U_{r,n}$, so $S$ is a basis of $\mathscr{C}_*(U_{r,n})$. 
		
		Finally for $(4)$ we specialize  \autoref{eq:bases of nice correspondence}. If $S\subseteq E$ and $\mathbf{B}_S$ is the matroid on $E$ with only one basis $S$, then
		\[
		\mathcal{B}_{\mathscr{C}_*(\mathbf{B}_S)}=\{T\subseteq E^*\mid (E\setminus S)\cup T\in \mathcal{B}_{\mathscr{C}}\}.
		\]
		Thus, we have
		\[
		\mathcal{B}_{\mathscr{C}}=\bigcup_{S\subseteq E} \{(E\setminus S)\cup T\mid T\in \mathcal{B}_{\mathscr{C}_*(\mathbf{B}_S)}\}.
		\]
		In particular, $\mathcal{B}_{\mathscr{C}}$ is determined by the images $\mathscr{C}_*(\mathbf{B}_S)$.
	\end{proof}
	
	\begin{remark}
		The four implications of the proposition are false in general for a   matroid $\mathscr{C}$ on $E\sqcup E^*$. For (1), (2), and (3), consider  \autoref{ex:weak maps not preserved}. In that example, $\rk(N) = 1 < 2 = \rk(\mathscr{C}_*(N))$, weak maps are not preserved, and $U_{0,2} \neq \mathscr{C}_*(U_{0,2})$.

		For statement (4), switching a coloop of $\mathscr{C}$ contained in $E$ to a loop or vice versa preserves the image of all single basis matroids, but in general does not preserve the others. 
		For example, let $E_1=\{1,2\},$ $E_2=\{1^*,2^*\}$, and compare the matroids $\mathscr{C}$ and $\mathscr{C}'$ defined by  $$\mathcal{B}_{\mathscr{C}}=\{\{2,1^*\},\{2,2^*\},\{1^*,2^*\}\}$$ and $$\mathcal{B}_{\mathscr{C}'}=\{\{1,2,1^*\},\{1,2,2^*\},\{1,1^*,2^*\}\}.$$ Then $\mathscr{C}_*(M)=\mathscr{C}'_*(M)$ for all matroids $M$ on $E_1$ with a single basis, but $\mathscr{C}_*(U_{1,2})\neq \mathscr{C'}_*(U_{1,2})$.  %Thus only knowing the images of single basis matroids is not enough to determine the whole map. 
	\end{remark}

	\begin{example}
		One might wonder whether every functor satisfying the conclusions of
		\autoref{prop:uniform_preserving_corrs} (1)-(3) must arise from a correspondence. This is not
		the case as the following example shows.
		
		Let \(E=\{1,2,3\}\), and let \(L\) be the rank-one matroid on \(E^*\) with bases
		\[
		\{\{2^*\},\{3^*\}\}.
		\]
		Thus \(1^*\) is a loop in \(L\). So $L^\vee$ is the rank-two
		matroid on \(E^*\) with bases
		\[
		\{\{1^*,2^*\},\{1^*,3^*\}\}.
		\]
		Define a functor
		\[
		\phi:\operatorname{Mat}(E)\longrightarrow \operatorname{Mat}(E^*)
		\]
		by
		\[
		\phi(M)=
		\begin{cases}
			U_{0,3}, & \operatorname{rk}(M)=0,\\
			U_{1,3}, & M=U_{1,3},\\
			L, & \operatorname{rk}(M)=1 \text{ and } M\neq U_{1,3},\\
			U_{2,3}, & M=U_{2,3},\\
			L^\vee, & \operatorname{rk}(M)=2 \text{ and } M\neq U_{2,3},\\
			U_{3,3}, & \operatorname{rk}(M)=3.
		\end{cases}
		\]
		One checks that \(\phi\) preserves matroid quotients,
		preserves rank, preserves weak maps, and fixes all uniform matroids.
		
		Suppose, for contradiction, that \(\phi=\mathscr{C}_*\) for some matroid \(\mathscr{C}\) on
		\(E\sqcup E^*\). By \autoref{prop:uniform_preserving_corrs} (4), the bases of \(\mathscr{C}\) are
		determined by the images of the one-basis matroids. 
		%More explicitly, if \(B_S\) denotes the matroid on \(E\) with unique basis \(S\), then
		%\[
		%\mathcal B(C)=
		%\bigcup_{S\subseteq E}
		%\{(E\setminus S)\cup T : T\in \mathcal B(\phi(B_S))\}.
		%\]
		In particular, since \(\phi(U_{0,3})=U_{0,3}\), the set \(E=\{1,2,3\}\) is a
		basis of \(\mathscr{C}\). 
		
		For $S\subset E$ let $\mathbf{B}_{S}$ be the matroid on $E$ with only one basis $S$. 
		Since \(\phi(\mathbf{B}_{\{1,2\}})=L^\vee\), the set
		$\{3,1^*,2^*\}$
		is also a basis of \(\mathscr{C}\). 
		Now apply the basis exchange axiom to the two bases
		$\{3,1^*,2^*\}$
		and
		$\{1,2,3\}$.
		Removing \(2^*\) from the basis $\{3,1^*,2^*\}$, the basis exchange axiom requires that either
		$\{3,1^*,1\}$
		or
		$\{3,1^*,2\}$
		is a basis of \(\mathscr{C}\). But since $\phi(\mathbf{B}_{\{1\}})=\phi(\mathbf{B}_{\{2\}})=L$ , neither of these sets is a basis of \(\mathscr{C}\). This contradiction shows that \(\phi\) is not induced by a matroid correspondence.
		
		To summarize, we have the following proper inclusions of functors $\Mat(E) \to \Mat(E^*)$:
		\begin{align*}
			\left\{\parbox{8em}{Correspondences preserving $U_{0,n}$ and $U_{n,n}$}\right\} \subsetneq\left\{\parbox{10em}{Functors preserving rank, weak maps, and uniform matroids}\right\} \subsetneq \left\{\text{Functors}\right\}.
		\end{align*}
	\end{example}

	\begin{remark}
		Matroid duality cannot be realized by matroid correspondences since it reverses matroid quotients instead of preserving them. Nevertheless, %following the principle of  \autoref{prop:uniform_preserving_corrs} (4), 
		if we were to construct a subset of $2^{E\sqcup E^*}$ to realize duality (not necessarily satisfying the basis exchange axiom), the subset would be of the form
		\[
		\{S\cup S^*\mid S\subseteq E\},
		\]
		which is the feasible set of a Delta-matroid. We expect that there is a notion of Delta matroid correspondence defined using Delta matroid union \cite[Corollary 6.2]{DeltaMatroidUnion} and contractions. We will not pursue this direction further here.
	\end{remark}

	\section{Background on polymatroids}\label{sec:background on polymatroids}
	\begin{definition}
		A \emph{polymatroid} $P$ on a finite set $E = \{1, \ldots, n\}$ is the data of a \emph{rank} function $\rk: 2^E \to \Z_{\geq 0}$, i.e., $\rk$ is
		\begin{description}
			\item[normalized] $\rk(\emptyset) = 0$,
			\item[increasing] if $A \subseteq B$, then $\rk(A) \leq \rk(B)$, and
			\item[submodular] if $A, B \subseteq E$, then $\rk(A \cup B) + \rk(A \cap B) \leq \rk(A) + \rk(B)$.
		\end{description}
		
		The \emph{rank} of $P$ is $\rk(P) \coloneqq  \rk(E)$. 
		A \emph{basis} of $P$ is a multiset $ \beta \in \N^E$ satisfying $|\beta| = \rk(P)$ and $|\pi_A(\beta)| \leq \rk(A)$ for all $A \subsetneq E$, where $\pi_A$ denotes the projection, $\pi_A: \N^E\to \N^A$. We denote the set of bases of $P$ by $\mathcal{B}_P$. 
		
		A \emph{cage} for a polymatroid $P$ on $E$ is $\alpha \in \N^E$ such that $\rk(i) \leq \alpha_i$ for all $i \in E$.
		The pair $(P, \alpha)$ is a \emph{caged polymatroid}. 
	\end{definition}
	
	\begin{definition}
		An \emph{M-convex set} in  \[
		\Delta^r_n\coloneq \{\alpha\in \Z_{\geq 0}^n\mid \alpha_1+\cdots +\alpha_n=r\},
		\]
		is a nonempty subset $N\subset \Delta^r_n$ such that for every $u,v\in N$ and $i\in \{1,\ldots, n\}$ one has: if $u_i<v_i$, then there exists $j\in \{1,\ldots, n\}$ such that $u_j>v_j$ and $u-e_j+e_i, v-e_i+e_j\in N$, where $\{e_1,\ldots, e_n\}$ is the standard basis of $\Z^n$.
	\end{definition}

	\begin{remark}
		Given any polymatroid $P$, $\mathcal{B}_P$ is always an M-convex set, see \cite[Subsection 1.2]{BHK+25}). Conversely, given an M-convex set in 
		$
		\Delta^r_n
		$, 
		there exists a unique polymatroid $P$ such that $\mathcal{B}_P$ is equal to the given M-convex set.
	\end{remark}
	
	We now define some basic operations of polymatroids analogous to those of matroids. 
	
	\begin{definition}
		For a subset $D\subseteq E$, the \emph{deletion} of $P$ by $D$, denoted by $P\setminus D$ is the polymatroid on $E\setminus D$ whose rank function is the restriction of the rank function of $P$ to subsets of $E\setminus D$. 
		
		Let $E_1$ and $E_2$ be disjoint finite sets. If $P_i$ is a polymatroid on $E_i$ for $i=1,2$, then their \emph{direct sum}, denoted by $P_1\oplus P_2$, is the polymatroid on $E_1\cup E_2$ whose rank function is defined by 
		\[
		\rk_{P_1\oplus P_2}(S_1\cup S_2)=\rk_{P_1}(S_1)+\rk_{P_2}(S_2)
		\]
		where $S_i\subseteq E_i$ for $i=1, 2$. 
		
		Let $(P, \alpha)$ be a caged polymatroid. It follows from the definition that $\{\alpha-B\mid B\in \mathcal{B}_P\}$ is also an M-convex set. Furthermore, the corresponding polymatroid also has $\alpha$ as a cage. This polymatroid is denoted by $P^\vee$, and $(P^\vee, \alpha)$ is called the \emph{dual caged polymatroid of} $(P, \alpha)$; see \cite[Section 3.1]{EL24} for more discussion about polymatroid duality.
		
		If $(P, \alpha)$ is a caged polymatroid on $E$, then the deletion $P\setminus D$ inherits a natural cage given by the restriction $\alpha|_{E\setminus D}$. Similarly, if $(P_1, \alpha_1)$ and $(P_2, \alpha_2)$ are caged polymatroids, then their direct sum $P_1\oplus P_2$ has a natural cage $(\alpha_1, \alpha_2)$. Consequently, the notions of deletion and direct sum extend naturally to caged polymatroids.
	\end{definition}

	In \cite{CHL+22} and \cite{CSW}, it is proved that there is a natural bijection between caged polymatroids and multisymmetric matroids. 
	
	\begin{definition}
		Let $\widetilde E = E^1 \sqcup \cdots \sqcup E^n$ be a finite set. A \emph{multisymmetric matroid} is a matroid $M$ on $\widetilde E$ whose rank function is invariant under the natural action of the product of the permutation groups $\prod_{i=1}^n \mathfrak S_{E^i}$.
	\end{definition}
	
	\begin{thmdfn}\label{lift}{\cite[Theorem/Definition 2.6]{CSW}}
		Fix a multiset $\alpha = (\alpha_1, \ldots, \alpha_n)$, and let $\widetilde E = E^1 \sqcup \cdots \sqcup E^n$ be a set with $|E^i| = \alpha_i$.
		There is a bijection
		\begin{align*}
			\{\text{caged polymatroids $(P, \alpha)$}\} &\overset{\sim}{\longrightarrow} \{\text{multisymmetric matroids on $\widetilde E$}\} \\
			(P, \alpha) &\longmapsto \widetilde P,
		\end{align*}
		where $\widetilde P$ is the \emph{multisymmetric lift} of $P$, defined by
		\[
		\rk_{\widetilde P}(A) \coloneqq \min \{\rk_P(B) + |A \setminus \cup_{i \in B} E^i| \mid B \subseteq E \}, \quad A \subseteq \widetilde E.
		\]
	\end{thmdfn}
	
	\begin{example}\label{ex:boolean_poly}
		A caged polymatroid $(P, \alpha)$ on $E$ is called \emph{Boolean}, if for any $S\subseteq E$, $\rk(S)=\sum_{i\in S}\alpha_i$. A caged polymatroid is Boolean if and only if its multisymmetric lift is a Boolean matroid. 
	\end{example}

	%\section{Polymatroid intersection}
	Next, we introduce polymatroid intersections, which  can only be defined for polymatroids sharing the same cage. 
	
	Given a multiset $\alpha\in \Z_{\geq 0}^n$ and subsets $\beta, \gamma\leq \alpha$, that is, $\beta, \gamma\in \Z_{\geq 0}^n$ satisfying $\beta_i, \gamma_i\leq \alpha_i$ for any $1\leq i\leq n$. We define the \emph{excess intersection} of $\beta$ and $\gamma$ relative to $\alpha$, denoted by $\beta\cap_{\alpha}\gamma$, by
	\[
	(\beta\cap_{\alpha}\gamma)_i=\max\{\beta_i+\gamma_i-\alpha_i, 0\}\quad \text{for any $1\leq i\leq n$.}
	\]
	\begin{thmdfn}\label{thmdfn:polymatroid intersection}
		Given $\alpha\in \Z_{\geq 0}^n$, and caged polymatroids $(P, \alpha)$ and $(Q, \alpha)$, the minimal multisets in the collection
		\[
		\{\beta_P\cap_\alpha \beta_Q \mid \beta_P\in \mathcal{B}_P \;\;\text{and}\;\; \beta_Q\in \mathcal{B}_Q\}
		\]
		form an M-convex set. The corresponding polymatroid is called the \emph{intersection} of caged polymatroids $(P, \alpha)$ and $(Q, \alpha)$, and denoted by $(P, \alpha)\wedge
		(Q, \alpha)$ or $P\wedge_\alpha Q$. The multiset $\alpha$ is the default cage of $P\wedge_\alpha Q$. 
	\end{thmdfn}
	\begin{proof}
		This is essentially proved in \cite{EL24}. 
		In \cite[Page 4226]{EL24}, it is shown that the basis polytope of a union of two polymatroids is M-convex. 
		By duality, it follows that the basis polytope of the intersection of two polymatroids with the same cage is M-convex. 
	\end{proof}

	When $\alpha=(1, \dots, 1)$, an $\alpha$-caged polymatroid is the same as a matroid, and the intersection as caged polymatroids is equal to matroid intersection.

	\begin{thmdfn}\label{thmdef:cap product}
		Let $P$ and $Q$ be polymatroids defined over the same set $E=\{1, \dots, n\}$. The minimal multisets in the collection
		\begin{equation}\label{eq:bases of cap product}
			\{B_Q-B_{P}\mid B_Q\in \mathcal{B}_Q\;\;\text{and}\;\; B_{P}\in \mathcal{B}_{P}\}
		\end{equation}
		form an M-convex set, where the subtraction is the \emph{multiset subtraction} defined by
		\[
		(B_Q-B_{P})_i=\max\{(B_Q)_i-(B_{P})_i, 0\}.
		\]
		This M-convex set defines a polymatroid, which is called the \emph{cap product} of $P$ and $Q$, and denoted by $P\frown Q$. 
	\end{thmdfn}
	\begin{proof}
		Choose a sufficiently large multiset $\alpha\in \Z_{\geq 0}^n$ which is a cage for both $P$ and $Q$. Notice that for $\beta, \gamma\in \Z_{\geq 0}^n$ with $\beta, \gamma\leq \alpha$, we have
		\[
		\beta\cap_\alpha \gamma=\beta - (\alpha -\gamma),
		\]
		where the subtractions are multiset subtractions. 
		Thus, the minimal multisets in \eqref{eq:bases of cap product} is equal to the set of bases of
		\[
		P^\vee\wedge_\alpha Q,
		\]
		where $(P^\vee, \alpha)$ is the dual caged polymatroid of $(P, \alpha)$. In particular, the minimal multisets in \eqref{eq:bases of cap product} form an M-convex set.
	\end{proof}
	
	In the following proposition we observe that taking caged polymatroid/matroid intersection commutes with multisymmetric lift.
	
	\begin{proposition}[{\cite[page 4226]{EL24}}]\label{prop:multisymmetric lift and intersection}
		Given two polymatroids $P$ and $Q$ with the same cage $\alpha$, we have the following equality  of multisymmetric matroids
		\[
		\widetilde{P\wedge_\alpha Q}= \widetilde{P}\wedge \widetilde{Q}. 
		\]
	\end{proposition}
	
	Similarly, we can define union of caged polymatroids using capped sum and show that the caged polymatroid/matroid union commutes with multisymmetric lift; see \cite[page 4226]{EL24}.

	\section{Lorentzian and volume polynomials}\label{sec:Lorentzian and volume}
	%Even though the theory of Lorentzian and volume polynomials is the main motivation of this paper, we will not get into the details of the definition of Lorentzian and volume polynomials. Instead, 
	In this section we  review the key properties of Lorentzian and volume polynomials that are relevant to this paper. 
	Introduced in \cite{BH20}, Lorentzian polynomials are polynomials  satisfying ``Hodge-Riemann conditions in degree one". Polymatroids should be regarded as discrete analogues of Lorentzian polynomials. 
	
	\begin{theorem}\cite[Theorem 3.10]{BH20}\label{thm:Lorentzian polymatroid}
		The support of any Lorentzian polynomial is the set of bases of a polymatroid. Conversely, the set of bases of any polymatroid is the support of some Lorentzian polynomial. 
	\end{theorem}

	A special class of Lorentzian polynomials are the \emph{realizable volume polynomials} over a fixed field $\mathbb{K}$. They are defined as the volume polynomial of a collection of base-point-free divisors on a projective variety. We refer to \cite{GHM+25} for more details about volume polynomials. Their combinatorial counterparts are the algebraic polymatroids over $\mathbb{K}$. The following result is analogous to  \autoref{thm:Lorentzian polymatroid}.
	
	\begin{proposition}\cite[Proposition 5.4]{GHM+25}\label{prop:volume polynomial polymatroid}
		Given a field $\mathbb{K}$, the support of any realizable volume polynomial over $\mathbb{K}$ is the set of bases of an algebraic polymatroid over $\mathbb{K}$. Conversely, the set of bases of any algebraic polymatroid over $\mathbb{K}$ is the support of some realizable volume polynomial over $\mathbb{K}$. 
	\end{proposition}
	
	Taking the associated multisymmetric matroid of a polymatroid is analogous to the polarization functor of volume polynomials. Given $\alpha\in \Z_{\geq 0}^n$, we define the total polarization operator by
	\begin{align*}
		\Pi^\uparrow_{\leq \alpha}: \R[x_1, \dots, x_n]_{\leq \alpha}&\to \R[x_{i,j}; 1\leq i\leq n, 1\leq j\leq \alpha_i]_{\leq \mathbf{1}}\\
		x_1^{k_1}\cdots x_n^{k_n}&\mapsto \prod_{1\leq i\leq n}\frac{e_{k_i}(x_{i,1}, \dots, x_{i, \alpha_i})}{\binom{\alpha_i}{k_i}}
	\end{align*}
	where $e_{k_i}(x_{i,1}, \dots, x_{i, \alpha_i})$ denotes the $k_i$-th elementary symmetric polynomial in $(x_{i,1}, \dots, x_{i, \alpha_i})$. Notice that the total polarization operator is the composition of the single-variable polarization operator defined in \cite[Section 4]{GHM+25} applied to each variable $x_i$. 
	
	The product of symmetric groups $\prod_{i=1}^n \mathfrak S_{\alpha_i}$ acts on the space %of multiaffine polynomials 
	$\R[x_{i,j}; 1\leq i\leq n, 1\leq j\leq \alpha_i]_{\leq \mathbf{1}}$ by permuting the variables. By definition, $\Pi_{\leq \alpha}^{\uparrow}(f)$ is invariant under this action. In particular, the support of $\Pi_{\leq \alpha}^{\uparrow}(f)$ is a subset of 
	\[
	\{(i, j)\mid 1\leq i\leq n, 1\leq j\leq \alpha_i\}
	\]
	that is preserved by the natural $\prod_{i=1}^n \mathfrak S_{\alpha_i}$-action. 
	
	The following lemma follows immediately from the definitions of the total polarization operator and the multisymmetric lift. 
	
	\begin{lemma}\label{lemma polarization}
		For fixed $\alpha\in \Z_{\geq 0}^n$, let $f\in \R[x_1, \dots, x_n]_{\leq \alpha}$ be a homogeneous polynomial whose support is the bases of a polymatroid $P$, e.g., $f$ is any Lorentzian polynomial. Then, $\alpha$ is  a cage of $P$ and the support of $\Pi_{\leq \alpha}^{\uparrow}(f)$ is equal to the multisymmetric lift of $(P, \alpha)$. 
	\end{lemma}
	
	\begin{proposition}\cite[Proposition 4.1]{GHM+25}\label{prop:polynomial polarization}
		For fixed $\alpha\in \Z_{\geq 0}^n$, let $f\in \R[x_1, \dots, x_n]_{\leq \alpha}$ be a homogeneous polynomial. Then, $f$ is Lorentzian if and only if $\Pi_{\leq \alpha}^{\uparrow}(f)$ is Lorentzian. Moreover, given a field $\mathbb{K}$, $f$ is a realizable volume polynomial over  $\mathbb{K}$ if and only if $\Pi_{\leq \alpha}^{\uparrow}(f)$ is a realizable volume polynomial  over $\mathbb{K}$. 
	\end{proposition}
	
	By \autoref{lemma polarization},  \autoref{prop:volume polynomial polymatroid}, and \autoref{prop:polynomial polarization}, we have the following corollary. 
	
	\begin{corollary}\label{cor:algebraic polymatroid lift}
		Given a polymatroid $P$ with ground set $E$ and a field $\mathbb{K}$, the following statements are equivalent.
		\begin{enumerate}
			\item The polymatroid $P$ is algebraic over $\mathbb{K}$.
			\item For any cage $\alpha$ of $P$, the multisymmetric lift of $(P, \alpha)$ is an algebraic matroid over $\mathbb{K}$.
			\item There exists a cage $\alpha$ of $P$ such that the multisymmetric lift of $(P, \alpha)$ is an algebraic matroid over $\mathbb{K}$.
		\end{enumerate}
	\end{corollary}

	As a consequence, in the following proposition we obtain that the intersection of caged algebraic polymatroids  is algebraic.
	
	\begin{proposition}\label{prop:polymatroid intersection preserves algebraic}
		%	Given a field $\mathbb{K}$, the intersection of algebraic polymatroids over $\mathbb{K}$ is algebraic over $\mathbb{K}$. More precisely, 
		Let $P$ and $Q$ be algebraic polymatroids over $\mathbb{K}$ and let $\alpha$ be a common cage of both $P$ and $Q$. Then, $P\wedge_\alpha Q$ is algebraic over $\mathbb{K}$. 
	\end{proposition}
	\begin{proof}
		Since $P$ and $Q$ are both algebraic over $\mathbb{K}$, by \autoref{cor:algebraic polymatroid lift}, their multisymmetric lifts $\widetilde{P}$ and $\widetilde{Q}$ are both algebraic matroids over $\mathbb{K}$. By \cite[Theorem 5.11]{GHM+25}, the matroid intersection $\widetilde{P}\wedge \widetilde{Q}$ is algebraic over $\mathbb{K}$. By  \autoref{prop:multisymmetric lift and intersection}, the multisymmetric lift $\widetilde{P \wedge_\alpha Q}$ is algebraic over $\mathbb{K}$. Finally, by  \autoref{cor:algebraic polymatroid lift}, the polymatroid intersection $P \wedge_\alpha Q$ is algebraic over $\mathbb{K}$. 
	\end{proof}

	\section{Polymatroid correspondence}\label{sec:polynmatroid correspondences}
	In this section, we generalize matroid correspondences to caged polymatroids.
	\begin{definition}\label{defn_poly correspondence}
		Fix $m,n\in \Z_{>0}$, and let $E_1=\{1, \dots, m\}$ and $E_2=\{m+1, \dots, n+m\}$. Given $\alpha\in \Z_{\geq 0}^{m}$ and $\beta\in \Z_{\geq 0}^{n}$, let $(\mathscr{C}, (\alpha, \beta))$ be a caged polymatroid on $E_1\cup E_2$. The \emph{polymatroid correspondence} induced by $\mathscr{C}$ is defined as
		\begin{align*}
			\mathscr{C}_*: \{\text{polymatroids with cage $\alpha$}\}&\to \{\text{polymatroids with cage $\beta$}\}\\
			P&\mapsto \big((P\oplus \mathbf{B}_\beta)\wedge_{(\alpha,\beta)} \mathscr{C}\big)\setminus E_1
		\end{align*}
		where $\mathbf{B}_\beta$ is the Boolean polymatroid with cage $\beta$ (see \autoref{ex:boolean_poly}). If we regard the set of polymatroids with a fixed cage as a poset category whose morphisms are polymatroid quotients, then $\mathscr C_*$ is a functor.
	\end{definition}
	
	To show that polymatroid correspondences preserve polymatroid quotients, we note that
	\begin{align*}
		P\twoheadrightarrow Q\text{ is a polymatroid quotient}&\iff \widetilde{P}\twoheadrightarrow \widetilde{Q} \text{ is a matroid quotient} \\
		&\implies  \widetilde{\mathscr{C}}_*(\widetilde{P})\twoheadrightarrow \widetilde{\mathscr{C}}_*(\widetilde{Q}) \text{ is a matroid quotient}\\
		&\iff  \widetilde{\mathscr{C}_*(P)}\twoheadrightarrow \widetilde{\mathscr{C}_*(Q)} \text{ is a matroid quotient}\\
		&\iff {\mathscr{C}_*(P)}\twoheadrightarrow {\mathscr{C}_*(Q)} \text{ is a matroid quotient},
	\end{align*}
	where the second equivalence follows from the next proposition,  a polymatroid analog of   \autoref{prop:polynomial polarization}.
	
	%We will not pursue the functoriality of the polymatroid correspondence map in this paper. Nevertheless, we will prove that polymatroid correspondence is compatible with multisymmetric lift and the polymatroid analog of Proposition \ref{prop:polynomial polarization}.
	
	\begin{proposition}\label{prop:commute with lift}
		Polymatroid correspondence is consistent with multisymmetric lift. More precisely, if $\alpha, \beta$, $\mathscr{C}$, and $P$ are as in  \autoref{defn_poly correspondence}, then 
		\[
		\widetilde{\mathscr{C}_*(P)}\cong \widetilde{\mathscr{C}}_*(\widetilde{P}). 
		\]
	\end{proposition}
	\begin{proof}
		Note that %It is straightforward to check that
		multisymmetric lift commutes with taking direct sum with $\mathbf{B}_\beta$ and with deletions. Moreover, by  \autoref{prop:multisymmetric lift and intersection}, multisymmetric lift also commutes with polymatroid intersection. Therefore, we have
		\begin{align*}
			\widetilde{\mathscr{C}_*(P)}&=\widetildespan{\big((P\oplus \mathbf{B}_\beta)\wedge_{(\alpha,\beta)} \mathscr{C}\big)\setminus E_1}\\
			&= \widetildespan{\big((P\oplus \mathbf{B}_\beta)\wedge_{(\alpha,\beta)} \mathscr{C}\big)}\setminus (E^1\cup\cdots \cup E^m)\\
			&= \big(\widetilde{P\oplus \mathbf{B}_\beta} \wedge \widetilde{\mathscr{C}}\big)\setminus (E^1\cup\cdots \cup E^m)\\
			&= \Big(\big(\widetilde{P}\oplus \widetilde{\mathbf{B}_\beta}\big) \wedge \widetilde{\mathscr{C}}\Big)\setminus (E^1\cup\cdots \cup E^m)\\
			&=\widetilde{\mathscr C}_*(\widetilde{P}),
		\end{align*}
		where the intersections in the third and fourth lines are intersections of matroids with ground set $\widetilde{E_1}\cup \widetilde{E_2}=E^1\cup \cdots \cup E^{m+n}$. 
	\end{proof}

	\begin{proposition}\label{prop:algebraic polymatroid correspondence}
		If $\mathscr{C}$ is an algebraic polymatroid over a field $\mathbb{K}$, then $\mathscr{C}_*$ maps algebraic polymatroids over $\mathbb{K}$ to algebraic polymatroids over $\mathbb{K}$. 
	\end{proposition}
	\begin{proof}
		%	By definition, the polymatroid correspondence is the composition of taking direct sum with $\mathbf{B}_\beta$, intersecting with $\mathscr{C}$, and deleting $E_1$. 
		We observe that taking direct sum with $\mathbf{B}_\beta$ and deletion preserve algebraic polymatroids. If $\mathscr{C}$ is algebraic over $\mathbb{K}$, then, by \autoref{prop:polymatroid intersection preserves algebraic}, taking intersection with $\mathscr{C}$ preserves algebraic polymatroids over $\mathbb{K}$. Therefore, if $\mathscr{C}$ is an algebraic polymatroid over $\mathbb{K}$, then $\mathscr{C}_*$ preserves algebraic polymatroids over $\mathbb{K}$. 
	\end{proof}
	
	We omit the proof of the following theorem as it is similar to that of its analog \autoref{thm:matroid example}. 
	
	\begin{theorem}
		Let $E=\{1,\ldots, n\}$ be a finite set and $(P,\alpha)$ a caged polymatroid on $E$. 
		The following functors on polymatroids  can be realized as correspondences via the given polymatroid. 
		We denote by $\mathbf{0}$ the vector $(0,\ldots, 0)\in \N^n$. 
		\begin{enumerate}
			\item \textbf{Identity:} 
			define the polymatroid $\mathscr{C}_{\mathrm{id}}$ on $E \sqcup E^*$ with cage $(\alpha, \alpha)$ and bases 
			\begin{align*}
				\{(\gamma, \alpha-\gamma) \mid \mathbf{0} \le \gamma \le \alpha\}.
			\end{align*}
			Then $\mathscr{C}_{\mathrm{id}*}(P)=P^*$.\smallskip

			\item \textbf{Permutation of the ground set:} for any permutation $\sigma$ of $E$ define the polymatroid $\mathscr{C}_{\mathrm{\sigma}}$ on $E \sqcup E^*$ with 
			with cage $(\alpha, \sigma(\alpha))$ and bases 
			\begin{align*}
				\{\left(\gamma, \sigma(\alpha) - \sigma(\gamma)\right) \mid \mathbf{0} \le \gamma \le \alpha\}.
			\end{align*}
			Then $\mathscr{C}_{\mathrm{\sigma}*}(P)=\mathrm{\sigma}(P)^*$, where  $\mathrm{\sigma}(P)$ is the polymatroid on $E$ whose bases are $\{\mathrm{\sigma}(B)\mid B\in \mathcal{B}_P\}$.\smallskip
			
			\item \textbf{Deletion:} 
			for any $e \in E$,  define the polymatroid $\mathscr{C}_{\setminus e}:= \mathscr{C}_{\mathrm{id}}\setminus e^*$ on $E \sqcup (E\setminus \{e\})^*$ with cage $(\alpha, \alpha_1, \ldots, \alpha_{i-1}, 0, \alpha_{i+1}, \ldots, \alpha_n)$ with bases 
			\begin{align*}
				\{(\gamma, \alpha-\gamma) \mid \mathbf{0} \le \gamma \le \alpha, \gamma_i = \alpha_i\}.
			\end{align*}
			Then $\mathscr{C}_{\setminus e *}(P)=(P\setminus e)^*$.\smallskip
			
			\item \textbf{Contraction}: 
			for any $e \in E$, define the polymatroid $\mathscr{C}_{/e} \coloneqq \mathscr{C}_{id}/e^*$ on $E \sqcup (E\setminus \{e\})^*$with cage $(\alpha, \alpha_1, \ldots, \alpha_{i-1}, 0 ,\alpha_{i+1}, \ldots, \alpha_N)$ and bases 
			\begin{align*}
				\{(\gamma, \alpha-\gamma) \mid \mathbf{0} \le \gamma \le \alpha, \gamma_i = \alpha_i\}.
			\end{align*}
			Then $\mathscr{C}_{/ e *}(P)=(P/e)^*$.\smallskip

			\item \textbf{Intersection with a fixed polymatroid:} 
			for any polymatroid $P_0$ on $E$ with cage $\alpha$ define the polymatroid $\mathscr{C}_{\wedge P_0}$ on $E \sqcup E^*$   by
			\[
			\mathscr{C}_{\wedge P_0}\coloneqq \mathscr{C}_{\mathrm{id}}\wedge_{(\alpha,\alpha)} (P_0 \oplus B_{\alpha}).
			\]
			Then $\mathscr{C}_{\wedge P_0*}(P)=(P\wedge_\alpha P_0)^*$.\smallskip

			\item \textbf{Union with a fixed polymatroid:} 
			for any polymatroid $P_0$ on $E$ with cage $\alpha$ define the polymatroid $\mathscr{C}_{\vee P_0}$ on $E \sqcup E^*$   by
			\[
			\mathscr{C}_{\vee P_0}\coloneqq \mathscr{C}_{\mathrm{id}}\vee_{(\alpha,\alpha)} (P_0 \oplus  P_{0,\alpha}),
			\]
			where $P_{0, \alpha}$ denotes the polymatroid of rank 0 with cage $\alpha$. 
			Then $\mathscr{C}_{\vee P_0*}(P)=(P\vee_\alpha P_0)^*$.\smallskip

			\item \textbf{Constant Map}: 
			let $E_1,E_2$ be finite sets. For any polymatroid $P_0$ on $E_2$ with cage  $\beta$, define 
			\begin{align*}
				\mathscr{C}_{P_0} \coloneqq \mathbf{B}_{\alpha}\oplus P_0.
			\end{align*}		
			Then $\mathscr{C}_{P_0*}(P)=P_0$ for every polymatroid $P$ on $E_1$ with cage $\alpha$.		
		\end{enumerate}
	\end{theorem}

	\section{Relation with symbols of linear operators}\label{sec:relation with symbols}
	In this section, we show that polymatroid correspondences are discrete analogues of symbols of linear operators. Then, we will deduce some realizability properties from the theory of volume polynomials. 
	
	First, we recall the definition of the symbol of a linear operator. Given $\alpha\in \Z_{\geq 0}^m$, we define
	\[
	\R[x]_{\leq \alpha}\coloneqq \operatorname{span}(x^\gamma)_{\gamma\leq \alpha}.
	\]
	Given $\alpha\in \Z_{\geq 0}^m$, $\beta\in \Z_{\geq 0}^n$, and a linear operator $T: \R[x]_{\leq \alpha}\to \R[y]_{\leq \beta}$, its \emph{symbol}, denoted by $\sym(T)$, is defined by
	\[
	\sym(T)\coloneqq \sum_{\gamma\leq \alpha}T(x^{[\gamma]})x^{[\alpha-\gamma]}
	\]
	where we use the divided power notation $x^{[\gamma]}\coloneqq \frac{x^{\gamma}}{\gamma!}=\prod_{1\leq i\leq m}\frac{x^{\gamma_i}}{\gamma_i!}$. We note that our  definition of symbol is the one in \cite{GHM+25}, which is different from the one in \cite{BH20} by a global positive constant. Since eventually we are only concerned with the support of polynomials, using either definition will not affect our results.
	
	Given $\alpha\in \Z_{\geq 0}^m$ and $f\in\R[x]_{\leq \alpha}$, the polynomial $f$ can be written uniquely as 
	\[
	f=\sum_{\gamma\leq \alpha}f_\gamma x^{[\gamma]}\in \R[x]_{\leq \alpha},
	\]
	where $f_\gamma\in \R$. The \emph{dual of $f$ with respect to $\alpha$}, denoted by $f^\vee$, is defined as
	\[
	f^\vee=\sum_{\gamma\leq \alpha}f_{\alpha-\gamma}\partial_x^{\gamma} \in \R[\partial_x]_{\leq \alpha}. 
	\]
	Notice that by definition, 
	\begin{equation}\label{eq:easy1}
		f^\vee(x^{[\alpha]})=f.
	\end{equation}
	Moreover, for any $f\in\R[x]_{\leq \alpha}$ we have the identity
	\begin{equation}\label{eq:easy2}
		f=\sum_{\gamma\leq \alpha}\partial_x^\gamma (f)\vert_{x=0}\cdot x^{[\gamma]}.
	\end{equation}
	%where $\big(x^{[\alpha-\gamma]}\big)^\vee$ is the dual with respect to $\alpha$. 
	\begin{lemma}\label{lemma:symbol to operator}
		The linear operator $T: \R[x]_{\leq \alpha}\to \R[y]_{\leq \beta}$ is uniquely determined by its symbol. More precisely, given $\sym(T)\in \R[x, y]_{\leq (\alpha, \beta)}$, let $\sym(T)^\vee\in \R[\partial_x, \partial_y]_{\leq (\alpha, \beta)}$ be the dual of $\sym(T)$ with respect to $(\alpha, \beta)$. Then, for any $f\in \R[x]_{\leq \alpha}$, 
		\[
		T(f)=\sym(T)^\vee(f\cdot y^\beta)|_{x=0}\in \R[y]_{\leq \beta}.
		\]
	\end{lemma}
	\begin{proof}
		By the definition of taking dual, we have
		\begin{equation}\label{eq:sym1}
			\sym(T)^\vee=\sum_{\gamma\leq \alpha}T\big(x^{[\gamma]}\big)^\vee\cdot \big(x^{[\alpha-\gamma]}\big)^\vee=\sum_{\gamma\leq \alpha}T\big(x^{[\gamma]}\big)^\vee \cdot \partial_x^\gamma,
		\end{equation}
		where the first dual on the right-hand side is with respect to $\beta$ and the second is with respect to $\alpha$. Therefore,
		\begin{align*}
			\sym(T)^\vee(f\cdot y^\beta)|_{x=0}&=\sum_{\gamma\leq \alpha}\left(T\big(x^{[\gamma]}\big)^\vee \cdot \partial_x^\gamma\right)(f\cdot y^\beta)\Big\vert_{x=0}\\
			&=\sum_{\gamma\leq \alpha}\big(\partial_x^\gamma(f)\vert_{x=0}\big)\cdot T\big(x^{[\gamma]}\big)^\vee(y^\beta)\\
			&=T\left(\sum_{\gamma\leq \alpha}\big(\partial_x^\gamma(f)\vert_{x=0}\big)x^{[\gamma]}\right)^\vee(y^\beta)\\
			&=T(f)^\vee(y^\beta)\\
			&=T(f),
		\end{align*}
		where the first equality follows from \eqref{eq:sym1}, the third equality follows from the fact that both $T$ and taking dual are linear operators, the fourth equality follows from \eqref{eq:easy2}, and the last equality follows from \eqref{eq:easy1}. 
	\end{proof}
	
	\begin{theorem}\cite[Theorem 3.2]{BH20}\label{thm:BH symbol}
		If $\sym(T)$ is a Lorentzian polynomial, then $T$ maps Lorentzian polynomials to Lorentzian polynomials. 
	\end{theorem}
	
	We demonstrate that the polymatroid correspondence is the discrete analog of the symbol. 
	
	Let $T: \R[x]_{\leq \alpha}\to \R[y]_{\leq \beta}$ be a linear operator such that $\sym(T)$ is Lorentzian. Let $f\in \R[x]_{\leq \alpha}$ be a Lorentzian polynomial such that $T(f)\neq 0$. By  \autoref{thm:BH symbol}, $T(f)$ is also a Lorentzian polynomial. 
	Denote the supports of $f$, $T(f)$, and $\sym(T)$ by $\mathcal{B}_f$, $\mathcal{B}_{T(f)}$, and $\mathcal{B}_T$, respectively. By \cite[Theorem~2.25]{BH20}, $\mathcal{B}_f$, $\mathcal{B}_{T(f)}$, and $\mathcal{B}_T$ are bases of polymatroids, and we denote the corresponding polymatroids by $P$, $Q$, and $\mathscr{C}$. By definition, $P$, $Q$, and $\mathscr{C}$ have natural cages $\alpha$, $\beta$ and $(\alpha, \beta)$, respectively. 
	\begin{theorem}\label{prop:polymatroid correspondence as support}
		Under the above notations, let $\mathscr{C}_*$ be the caged polymatroid correspondence defined as in  \autoref{defn_poly correspondence}. Then, as caged polymatroids
		\[
		(Q, \beta)=(\mathscr{C}_*(P), \beta).
		\]
	\end{theorem}
	\begin{proof}
		We denote by $\Supp(g)$ the support of the polynomial $g$. We need show that  $\Supp(T(f)) = \mathcal{B}_{\mathscr{C}_*(P)}$. 
		
		Since  $\Supp(\sym(T))= \mathcal{B}_{\mathscr{C}}$, we have   $\Supp(\sym(T)^\vee)=\mathcal{B}_{\mathscr{C}^\vee}$, where  $\sym(T)^\vee$ is the dual with respect to $(\alpha, \beta)$ and  $\mathscr{C}^\vee$ is the dual with respect  to  $(\alpha, \beta)$. 
		Since $\Supp(f)=\mathcal{B}_P$, we have that  $\Supp(f\cdot y^\beta)=\mathcal{B}_{P \oplus \mathbf{B}_{\beta}}$, and from the definition of polymatroid cap product in \autoref{thmdef:cap product}, it follows that %since $\sym(T)^\vee(f\cdot y^\beta)\neq 0$, 
		$\Supp(\sym(T)^\vee(f\cdot y^\beta))$ is the set of bases of  $\mathscr{C}^\vee\frown (P\oplus \mathbf{B}_{\beta})$. 
		
		In the proof of  \autoref{thmdef:cap product}, we  proved the equality 
		\[
		\mathscr{C}^\vee\frown (P\oplus \mathbf{B}_{\beta})= \mathscr{C}\wedge_{(\alpha, \beta)}(P\oplus \mathbf{B}_{\beta}).
		\]
		Thus,  $\Supp(\sym(T)^\vee(f\cdot y^\beta))$ is the set of bases of  $\mathscr{C}\wedge_{(\alpha, \beta)}(P\oplus \mathbf{B}_{\beta})$. Therefore,  $\Supp(\sym(T)^\vee(f\cdot y^\beta)|_{x=0})$ is  the set of bases of $\mathscr{C}\wedge_{(\alpha, \beta)}(P\oplus \mathbf{B}_{\beta})\setminus E_1$, and the latter  is equal to $\mathscr{C}_*(P)$  by definition. The conclusion now follows from 
		\autoref{lemma:symbol to operator}.
	\end{proof}
	
	\bibliographystyle{alpha} 
	\bibliography{refs}
\end{document}